\documentclass[11pt,a4paper]{article}
\usepackage{amssymb,amsmath}
\usepackage{color}
\usepackage{a4wide}
\usepackage{xcolor}
\usepackage{amssymb}
\usepackage{exscale,cmmib57}
\usepackage{stmaryrd}

\numberwithin{equation}{section}
\newcommand {\eq} [1] {\begin{equation}\label{#1}}
\newcommand {\en} {\end{equation}}
\newcommand {\eqn}  {\begin{eqnarray}}
\newcommand {\enn} {\end{eqnarray}}
\newcommand {\bstar}  {\begin{eqnarray*}}
\newcommand {\estar} {\end{eqnarray*}}

\newcommand {\setC}  {{\mathbb C}}

\newcommand {\mat} [1] {\left[\begin{array}{#1}}
\newcommand {\rix}     {\end{array}\right]}
\newtheorem{Theorem}{Theorem}[section]
\newtheorem{Lemma}{Lemma}[section]

\newtheorem{Corollary}{Corollary}[section]

\newtheorem{Remark}{{\it Remark}}[section]

\newcommand {\wh} {\widehat}
\newcommand {\wt} {\widetilde}

\newcommand {\proof} {\par{\it Proof}. \ignorespaces}
\newcommand {\eproof}
      {\space
 {\ \vbox{\hrule\hbox{\vrule height1.3ex\hskip0.8ex\vrule}\hrule}}
 \par}
\newcommand {\diag}  {\mathop{\rm diag}\nolimits}

\title{Combined perturbation bounds for eigenstructure of Hermitian matrices and singular structure of general matrices\thanks{This work is supported by National Natural Science Foundations of China (11771159) and Guangdong Provincial Natural Science Foundation (2022A1515011123)}}
\author{Xiao Shan Chen\footnotemark[1]
\and Hongguo Xu\footnotemark[2]}

\begin{document}
\maketitle
\renewcommand{\thefootnote}{\fnsymbol{footnote}}
\footnotetext[1]
{School of Mathematics, South China
Normal University, Guangzhou, 510631, People's Republic of China. \texttt{chenxs33@163.com}.
This work is supported by National Natural Science Foundations of China (11771159, 11671158)  and
and Guangdong Provincial Natural Science Foundation (2022A1515011123).
}

\footnotetext[2]{ Department of Mathematics, University of Kansas, Lawrence, KS 66045, USA.
\texttt{feng@ku.edu}. }
\renewcommand{\thefootnote}{\arabic{footnote}}

\begin{abstract}
Combined perturbation bounds are presented for eigenvalues and eigenspaces
of Hermitian matrices or singular values and singular subspaces of  general matrices.
The bounds are derived based on the smooth decompositions and elementary calculus techniques.
\end{abstract}

\noindent
{\bf Keywords}
combined perturbation bound, eigenvalue, eigenspace,  singular value, singular subspace, Frobenius norm

\noindent
{\bf AMS subject classification}
65F15, 65F99

\section{Introduction} Throughout this paper the symbol $\mathbb{C}^{m\times n}$ ($\mathbb{R}^{m\times n}$) denotes the set of complex (real) $m\times n$ matrices. $A^H (A^T)$ is the conjugate transpose (transpose) of a matrix $A$. The $(i,j)$ entry of a matrix $A$ is denoted by $A_{ij}$. The set of singular values of a matrix $A$ is denoted by $\sigma(A)$. For a square matrix $A$, the spectrum of $A$ is denoted by $\lambda(A)$.
For an $n\times n$ Hermitian matrix $A$, ${\rm Eig}^{\downarrow}(A)$ denotes an $n\times n$ diagonal matrix whose diagonal entries are the eigenvalues of $A$ in nonincreasing order. For an $m\times n (m\geq n)$ matrix $B$, ${\rm Sing}^{\downarrow}(B)$ denotes an $n\times n$ diagonal matrix whose diagonal entries are the singular values of $B$ in nonincreasing order.
$\mathcal R (B)$ denotes the range of the matrix $B$.
$I_n$ (or simply $I$) is the identity matrix of order $n$ and $e_i$ is its $i$th column. $\|\cdot\|_2$ denotes the spectral matrix norm and $\|\cdot\|_F$ the Frobenius norm.
We use the notation $\dot{F}(t)$ for $dF(t)/dt$, where $F(t)$ can be a time-dependent scalar, vector, or  matrix. For a complex number $z$, by $\overline{z}$, ${\rm Re}(z)$ and ${\rm Im}(z)$ we denote its conjugate, real and imaginary parts, respectively. Finally ${\imath} =\sqrt{-1}$.
\vskip 0.1in
Perturbation theory about the eigenvalues and eigenspaces of Hermitian matrices,
and the singular values and singular subspaces of general matrices has been well established, and many results have been published; e.g., see \cite{dk70,golv13,hw53,kat80,li07,s73,ss90,w72}. The purpose of this paper is to study  perturbation bounds in a combined form for a Hermitian matrix by introducing a single real parameter to the perturbation matrix and considering a (general) spectral  factorization as an analytic form.
Another purpose is to apply the same technique to the singular value decomposition (SVD)
and establish the same type of results.

Let $A$ and $\wt A=A+\Delta A$ be two $n\times n$ Hermitian matrices. Then one has the
Hoffman--Wielandt type eigenvalue bound (\cite{bh96,golv13,ss90})
\begin{eqnarray}\label{c-121}
\|{\rm Eig}^{\downarrow}(\wt A)-{\rm Eig}^{\downarrow}(A)\|_F\leq \|\Delta A\|_F.
\end{eqnarray}
Let $\mathcal{R}(U_1)$ and $\mathcal{R}(\wt U_1)$ be $r$-dimensional eigenspaces of $A$ and $\wt A$, respectively, where $U_1^HU_1=\wt U_1^H\wt U_1=I_r$.
The canonical angle between $\mathcal{R}(U_1)$ and $\mathcal{R}(\wt U_1)$ is defined by
$$
\Theta(U_1,\wt U_1)=\arccos(U_1^H\wt U_1\wt U_1^HU_1)^{1/2}.
$$
Under the condition
\[
\delta_{12}=\min_{\lambda\in\lambda(U_1^HAU_1), \wt\lambda\in\lambda(\wt A)/\lambda(\wt U_1^H\wt A\wt U_1)}\{|\lambda-\wt \lambda|\}>0,
\]
Davis and Kahan \cite{dk70} provided the following classical perturbation bound for eigenspaces of Hermitian matrices
\begin{eqnarray}\label{c-122}
\|\sin\Theta(U_1,\wt U_1)\|_F\leq \frac{1}{\delta_{12}}\|\Delta A U_1\|_F.
\end{eqnarray}
In \cite{li07}, Li and Sun  obtained perturbation bounds
in a combined form for eigenspaces and the corresponding eigenvalues.
One of the bounds 
 (\cite[Theorem 2.2]{li07}) is
\begin{eqnarray}\label{c-13}
&&\quad(1-\|\sin\Theta(U_1,\wt U_1)\|_2^2)\|{\rm Eig}^{\downarrow}(\wt U_1^H \wt A \wt U_1)-{\rm Eig}^{\downarrow}(U_1^HAU_1)\|_F^2\nonumber\\
&&\qquad\qquad\qquad\qquad \qquad\quad +\delta^2_{12}\|\sin\Theta(U_1,\wt U_1)\|_F^2
\leq \|\Delta A U_1\|_F^2.
\end{eqnarray}
The bound (\ref{c-13}) is sharper than (\ref{c-122}). When $r=n$, one has
 $\sin\Theta(U_1, \wt U_1)=0$ and the inequality (\ref{c-13}) reduces to (\ref{c-121}). Similar perturbation bounds have been established for singular values and (left and right) singular subspaces of a general matrix   (\cite{li98a,li98, li07,ss90,w72}).  \vskip 0.1in

In this paper we will provide same types of combined bounds for the eigenvalues and
eigenspaces of a Hermitian matrix and the singular values and singular subspaces of
a general matrix.  The contributions of the work  can be summarized as follows, which
is for Hermitian eigenvalue problem only. It is similar for the SVD results.
\begin{enumerate}
\item[(a)] The techniques involved in \cite{bunbmn91,de99, kat80}  are essentially elementary calculus. This is different from ones in \cite{li07}, where advanced inequalities are employed;
\item[(b)] We derive novel local bounds for perturbation of eigenvalues and several eigenspaces and for one eigenspace and its corresponding eigenvalues;

\item[(c)] For measuring  perturbation of eigenspaces, instead of using the canonical angle we use the distance between two orthonormal basis matrices. As a consequence, the derived bound essentially implies a bound (\ref{c-13}), so it is potentially sharper.
\end{enumerate}\vskip 2mm

The rest of this paper is organized as follows. In Section 2 we present combined perturbation bounds for eigenspaces and corresponding eigenvalues of a Hermitian  matrix.
In Section 3 we derive combined perturbation bounds for singular subspaces and corresponding singular values of a general matrix. Section 4 contains our conclusions.

\section{Combined bounds of eigenvalues and eigenspaces}
In this section we derive combined perturbation bounds for eigenspaces and corresponding eigenvalues of a Hermitian matrix. The following result is essential for deriving our main results.\vskip 2mm

\begin{Lemma}~\label{lem21}
Suppose $U(t)=[U_1(t), \ldots, U_k(t)]$ is an $n\times n$ unitary analytic time-dependent matrix of a real variable $t$, where
 $$ U_j(t)\in\mathbb{C}^{n\times r_j},\quad j=1,\ldots,k,\quad\mbox{and} \quad r_1+\cdots+r_k=n.$$
 Then for any given skew Hermitian analytic time-dependent matrices $\Phi_j(t)\in\mathbb{C}^{r_j\times r_j}, j=1,\ldots,k$, there exist unitary analytic time-dependent matrices $P_j(t)\in\mathbb{C}^{r_j\times r_j}$,
$j=1,\ldots,k$, such that for $\wh U_j(t)=U_j(t)P_j(t)$ one has $\wh U_j(t)^H \dot {\wh U}_j(t)=\Phi_j(t)$ for $j=1,\ldots,k$.
\end{Lemma}\vskip 2mm

\proof
Taking the derivative on both sides of $U(t)^HU(t)=I$ leads to
\[
U(t)^H\dot U(t)=-\dot U(t)^H U(t)=-(U(t)^H\dot U(t))^H.
\]
Hence  $U_j(t)^H\dot U_j(t)$ is a skew Hermitian matrix for $j=1,\ldots, k.$\vskip 0.1in
Let $P(t)={\rm diag}(P_1(t),\ldots,P_{k}(t))$ and  $\wh U(t)=U(t)P(t)$. Then
\begin{eqnarray*}
\wh U(t)^H\dot{\wh U}(t)&=&P(t)^HU(t)^H\left[\dot U(t) P(t)+U(t)\dot P(t)\right]\\
&=&P(t)^HU(t)^H\dot U(t) P(t)+P(t)^H\dot P(t),
\end{eqnarray*}
which can be written as
$$
\dot P(t)=P(t)\left[\wh U(t)^H\dot{\wh U}(t)-P(t)^H U(t)^H\dot U(t) P(t)\right].
$$
By comparing the $j$-th diagonal block on both sides of the above equation, $P_{j}(t)$ has to satisfy the differential equation:

\[\dot P_{j}(t)=P_{j}(t) \left[\Phi_{j}(t)-P_j(t)^H U_j(t)^H\dot U_j(t) P_j(t)\right] ,\quad j=1,\ldots,k.
\]
Since $\Phi_{j}(t)-P_j(t)^HU_j(t)^H\dot U_j(t)P_j(t)$ is skew Hermitian as a sum of skew Hermitian matrices with the
initial condition $P_j(0)=I$, the differential equation has unique solution that is unitary
and analytic \cite{de99}. This shows the existence of $P_j(t), j=1,\ldots, k.$
\eproof
\medskip

Let $A(t)=A+t\Delta A\in\mathbb{C}^{n\times n}$ be a Hermitian matrix with $t\in{\mathbb R}$. Then
it is known (\cite{gohlr09,kat80,rel37,rel69}) that $A(t)$ has an analytic decomposition
\begin{eqnarray}\label{spectral}
A(t)=U(t)\Lambda(t)U(t)^H,
\end{eqnarray}
where
\begin{eqnarray}\label{bku}
U(t)=\left[U_1(t),\ldots,U_k(t)\right],\qquad U_j(t)\in\setC^{n\times r_j}, \quad j=1,\ldots,k,
\end{eqnarray}
is unitary and analytic,
and
\begin{eqnarray}\label{bkl}
\Lambda(t)=\diag(\Lambda_1(t),\ldots,\Lambda_k(t)),\qquad
\Lambda_j(t)\in{\setC}^{r_j\times r_j},\quad j=1,\ldots,k
\end{eqnarray}
is analytic.  Note that $\Lambda_1(t),\ldots,\Lambda_k(t)$ may or may not be diagonal.

The analytic decomposition (\ref{spectral}) can be considered as a generalized spectral decomposition and it is not unique. We have the following
results with a special choice of $U(t)$.\vskip 2mm

\begin{Lemma}\label{lem22}
Let $A(t)=A+t\Delta A\in\mathbb{C}^{n\times n}$ be a Hermitian matrix with $t\in{\mathbb R}$. Then
$A(t)$ has the analytic decomposition (\ref{spectral}) with
$U(t)$ and $\Lambda(t)$ in the block forms (\ref{bku}) and (\ref{bkl}), and $U(t)$
satisfies
\begin{eqnarray}\label{c-2-3}
U_j(t)^H\dot U_j(t)=0,\quad j=1,\ldots,k.
\end{eqnarray}
\end{Lemma}
\begin{proof}
Let $A(t)$ have an analytic decomposition (\ref{spectral}). Then for
arbitrary unitary analytic  matrices  $P_j(t)\in{\setC}^{r_j\times r_j}$, $j=1,\ldots,k$,
we have
\[
A(t)=\wh U(t)\wh\Lambda(t)\wh U(t)^H,
\]
where
\[
\wh U(t)=U(t)\diag(P_1(t),\ldots,P_k(t)),
\quad \wh \Lambda(t)=\diag(P_1(t)^H\Lambda_1(t)P_1(t),\ldots,
P_k(t)^H\Lambda_k(t)P_k(t)).
\]
By Lemma~\ref{lem21} with $\Phi_j(t)=zeros(r_j, r_j)$, $P_1(t),\ldots,P_k(t)$ can be selected such that
\[
\wh U_j(t)^H\dot {\wh U}_j(t)=0,\quad j=1,\dots,k.
\]
Hence $A(t)$ has an analytic  decomposition (\ref{spectral}) that satisfies   (\ref{c-2-3}).
\end{proof}

\medskip

We have  the following perturbation result of the eigenspaces ${\mathcal R}(U_j(0))$,
$j=1,\ldots,k$, and the eigenvalues of $A=A(0)$.\vskip 2mm

\begin{Theorem}\label{them1}
Let $A(t)=A+t\Delta A\in\mathbb{C}^{n\times n}$ be a Hermitian matrix with $t\in{\mathbb R}$.
Suppose $A(t)$ has the analytic decomposition (\ref{spectral}) with $U(t)$ satisfying
(\ref{bku}) and (\ref{c-2-3}) and $\Lambda(t)$ in the block diagonal form (\ref{bkl}).
Define
\begin{eqnarray}\label{c-2-1}
\delta_{ji}(t)=\min_{\lambda_1(t)\in\lambda(\Lambda_j(t)),~
\lambda_2(t)\in\lambda(\Lambda_i(t))}\left\{\left|\lambda_1(t)-\lambda_2(t)\right|\right\},\quad
\delta_j(t)=\min_{i\ne j}\left\{\delta_{ji}(t)\right\},
\end{eqnarray}
and
\begin{eqnarray}\label{c-2-2}
\delta_{j,\min}=\min_{0\le t\le 1}\delta_j(t),
\end{eqnarray}
for $j=1,\ldots,k$. Denote $\wt A=: A(1)=A+\Delta A$ and
\begin{eqnarray*}
&&U(0)=\left[U_1(0),\ldots, U_k(0)]=:[U_1, \ldots, U_k\right]=U,\\
&&U(1)=\left[U_1(1),\ldots, U_k(1)\right]=:[\wt U_1,\ldots, \wt U_k ]=\wt U. 
\end{eqnarray*}
Then
\begin{eqnarray}\label{c-2-4}
\|{\rm Eig}^{\downarrow}(\wt A)-{\rm Eig}^{\downarrow}(A)\|_F^2+\sum_{j=1}^k
\delta_{j,\min}^2\|\wt U_j-U_j\|_F^2\le
\|{\Delta A}\|_F^2.
\end{eqnarray}
\end{Theorem}\vskip 2mm

\begin{proof}
By taking derivatives on both sides of (\ref{spectral}) and $U^H(t)U(t)=I,$ we have
\begin{eqnarray}\label{dspectral}
\Delta A=\dot{U}(t)\Lambda(t)U(t)^H+U(t)\dot{\Lambda}(t)U(t)^H+U(t)\Lambda(t)\dot{U}(t)^H
\end{eqnarray}
and
\begin{eqnarray}\label{dunitary}
\dot{U}(t)^H U(t)=-U(t)^H \dot{U}(t).
\end{eqnarray}
Multiplying both sides of (\ref{dspectral}) with $U(t)^H$ on the left and $U(t)$ on the right,
and using  (\ref{dunitary}), we obtain
\begin{eqnarray}\label{dspectral-2}
U(t)^H \Delta A U(t) =\dot{\Lambda}(t)+U(t)^H \dot{U}(t)\Lambda(t)-\Lambda(t) U(t)^H \dot{U}(t).
\end{eqnarray}
Let
\[
U_i(t)^H\Delta AU_j(t)=:[\Delta A_{ij}(t)],\quad i,j=1,\ldots,k.
\]
By comparing the blocks of (\ref{dspectral-2}) and using (\ref{c-2-3}), one obtain its diagonal terms
\begin{eqnarray}\label{c-2-5}
\dot\Lambda_j(t)=\Delta A_{jj}(t),\quad j=1,\ldots, k,
\end{eqnarray}
and off-diagonal terms
\begin{eqnarray}\label{c-2-6}
U_i(t)^H\dot U_j(t)\Lambda_j(t)-\Lambda_i(t)U_i(t)^H\dot U_j(t)=\Delta A_{ij}(t),
\qquad i\ne j.
\end{eqnarray}
Suppose that
\[
\Lambda_j(t)=G_j(t)\mat{ccc}\lambda_{j,1}(t)&&\\&\ddots&\\&&\lambda_{j,r_j}(t)\rix
G_j(t)^H,\quad j=1,\ldots,k,
\]
are  spectral decompositions (not necessarily analytic),
where $G_1(t),\ldots,G_k(t)$
are unitary matrices. Multiplying both sides of (\ref{c-2-6}) with $G_i(t)^H$ on the left and  $G_j(t)$ on the right yields
\begin{eqnarray}\label{c-2-7}
X_{ij}(t)\mat{ccc}\lambda_{j,1}(t)&&\\&\ddots&\\&&\lambda_{j,r_j}(t)\rix-
\mat{ccc}\lambda_{i,1}(t)&&\\&\ddots&\\&&\lambda_{i,r_i}(t)\rix X_{ij}(t)
=Z_{ij}(t),
\end{eqnarray}
where $X_{ij}(t)=G_i(t)^HU_i(t)^H\dot U_j(t)G_j(t)$
and $Z_{ij}(t)=G_i(t)^H\Delta A_{ij}(t)G_j(t)$. Let
$$X_{ij}(t)=[x^{(ij)}_{pq}(t)],\quad   Z_{ij}(t)=[z_{pq}^{(ij)}(t)].$$
 From (\ref{c-2-7}) one has
\begin{eqnarray*}
z^{(ij)}_{pq}(t)=(\lambda_{j,q}(t)-\lambda_{i,p}(t))x_{pq}^{(ij)}(t), \quad p=1,\ldots,r_i,\quad
 q=1,\ldots,r_j.
\end{eqnarray*}
Then for $\delta_{ji}(t)$ defined in (\ref{c-2-1}), one has
\begin{eqnarray}\label{c-2-8}
\|{\Delta A_{ij}(t)}\|_F^2&=&
\|{Z_{ij}(t)}\|_F^2=\sum_{p=1}^{r_i}\sum_{q=1}^{r_j}|z_{pq}^{(ij)}(t)|^2
=\sum_{p=1}^{r_i}\sum_{q=1}^{r_j}|\lambda_{j,q}(t)-\lambda_{i,p}(t)|^2|x_{pq}^{(ij)}(t)|^2
 \nonumber\\
&\ge& \delta_{ji}(t)^2\sum_{p=1}^{r_i}\sum_{q=1}^{r_j}|x_{pq}^{(ij)}(t)|^2
=\delta_{ji}(t)^2\|{U_i(t)^H\dot U_j(t)}\|_F^2.
\end{eqnarray}
By (\ref{c-2-3}),  we have
\begin{eqnarray}\label{c-2-9}
\|{\dot U_j(t)}\|_F^2=\|{U(t)^H\dot U_j(t)}\|_F^2=\sum_{i=1,i\ne j}^k\|{U_i(t)^H\dot U_j(t)}\|_F^2.
\end{eqnarray}
Then we get
\begin{eqnarray}\label{c-2-10}
\|{\Delta A}\|_F^2&=&\|{U(t)^H\Delta AU(t)}\|_F^2=\sum_{j=1}^k\|{\Delta A_{jj}(t)}\|_F^2
+\sum_{i\ne j}\|{\Delta A_{ij}(t)}\|_F^2\nonumber\\
&\ge &\sum_{j=1}^k\|{\dot\Lambda_j(t)}\|_F^2
+\sum_{j=1}^k\sum_{i=1,i\ne j}^k\delta_{ji}(t)^2\|{U_i(t)^H\dot U_j(t)}\|_F^2   \qquad (\mbox{by}\; (\ref{c-2-5})\; \mbox{and}\; (\ref{c-2-8}) )\nonumber\\
&\ge &\|{\dot\Lambda(t)}\|_F^2+\sum_{j=1}^k\delta_j(t)^2\sum_{i=1,i\ne j}^k\|{U_i(t)^H\dot U_j(t)}\|_F^2  \qquad (\mbox{by}\; (\ref{c-2-1}))\nonumber\\
&=&\|{\dot\Lambda(t)}\|_F^2
+\sum_{j=1}^k\delta_j(t)^2\|{\dot U_j(t)}\|_F^2\qquad (\mbox{by}\; (\ref{c-2-9}))
\nonumber\\
&\ge& \|{\dot\Lambda(t)}\|_F^2+\sum_{j=1}^k\delta_{j,\min}^2\|{\dot U_j(t)}\|_F^2.\qquad (\mbox{by}\; (\ref{c-2-2}))
\end{eqnarray}
Using the fact that
\[
\left|\int_0^1f(t){\rm d}t\right|^2\leq\int_0^1|f(t)|^2{\rm d}t,\quad \forall f(t),
\]
and together with (\ref{c-2-10}), we get
\begin{eqnarray}\label{c-2-11}
\|{\Delta A}\|_F^2&=&\int_0^1\|{\Delta A}\|_F^2{\rm d}t\geq
\int_0^1\|{\dot\Lambda(t)}\|_F^2{\rm d}t
+\sum_{j=1}^k\delta_{j,\min}^2\int_0^1\|{\dot U_j(t)}\|_F^2{\rm d}t\nonumber\\
&\geq&
\left\|\int_0^1\dot\Lambda(t){\rm d}t\right\|_F^2
+\sum_{j=1}^k\delta_{j,\min}^2\left\|\int_0^1\dot U_j(t){\rm d}t\right\|_F^2\nonumber\\
&=&\|{\Lambda(1)-\Lambda(0)}\|_F^2+\sum_{j=1}^k\delta_{j,\min}^2\|{U_j(1)-U_j(0)}\|_F^2
\nonumber\\
&=&\|{\Lambda(1)-\Lambda(0)}\|_F^2+\sum_{j=1}^k\delta_{j,\min}^2\|{\wt U_j-U_j}\|_F^2.
\end{eqnarray}
By (\ref{c-2-11}) and the inequality
\[
\|{\rm Eig}^{\downarrow}(\wt A)-{\rm Eig}^{\downarrow}(A)\|_F=\|{\rm Eig}^{\downarrow}(\Lambda(1))-{\rm Eig}^{\downarrow}(\Lambda(0))\|_F\leq\|{\Lambda(1)-\Lambda(0)}\|_F,
\]
which is from (\ref{c-121}),  we obtain (\ref{c-2-4}).
\end{proof}

\medskip

\begin{Remark}\label{rem211}
The combined perturbation bound (\ref{c-2-4}) is sharper than  (\ref{c-121}). When $k=1$, (\ref{c-2-3}) implies  $\wt U=U$. In this case, (\ref{c-2-4}) reduces to (\ref{c-121}).
\end{Remark}

\medskip

\begin{Remark}\label{rem24}
In order to include the term  $\|\wt U_j-U_j\|_F^2$ in (\ref{c-2-4}), we need $\delta_{j,\min}>0$, for which a sufficient condition is
 $$
 2\|\Delta A\|_2<\delta_j(0),
 $$
where $\delta_j(0)$ is given by (\ref{c-2-1}). In fact, for any $\lambda_1(t)\in\lambda(\Lambda_j(t))$ and $\lambda_2(t)\in\cup_{i\ne j}\lambda(\Lambda_i(t))$
with $t\in (0,~1]$ ,  there exist (\cite[Chapter IV, Corollary 4.10]{ss90}) $\mu\in\lambda(\Lambda_j(0))$ and $\nu\in\cup_{i\ne j}\lambda(\Lambda_i(0))$
such that
$$
|\lambda_1(t)-\mu|\leq \|\Delta A\|_2, \quad |\lambda_2(t)-\nu|\leq\|\Delta A\|_2.
$$
Therefore,
\begin{eqnarray*}
|\lambda_1(t)-\lambda_2(t)|\geq |\mu-\nu|-|\lambda_1(t)-\mu|-|\lambda_2(t)-\nu|\geq \delta_j(0)-2\|\Delta A\|_2,
\end{eqnarray*}
which implies
$$
\delta_{j,\min}\geq \delta_j(0)-2\|\Delta A\|_2>0.
$$
\end{Remark}

\medskip

\begin{Remark}\label{rem22}
When $k=n$ and $r_1=\cdots=r_n=1$, (\ref{spectral}) is a spectral decomposition.
All $U_1,\ldots,U_n$ and $\wt U_1,\ldots,\wt U_n$
are eigenvectors and (\ref{c-2-4}) bounds the perturbations of all the eigenvalues and
eigenvectors of $A$ (with the assumption that $\delta_{j,\min}>0$ for $j=1,\ldots,n$).
In particular, (\ref{c-2-4}) implies
\[
\|{\rm Eig}^{\downarrow}(\wt A)-{\rm Eig}^{\downarrow}(A)\|_F^2+\delta_{\min}^2
\|{\wt U-U}\|_F^2\le \|{\Delta A}\|_F^2,
\]
where $\delta_{\min}=\min_{1\le j\le k}\{\delta_{j,\min}\}$. When $\delta_{\min}>0$,
it bounds the perturbations of all the eigenvalues and the entire unitary similarity matrix $U$
of $A$.
\end{Remark}

\medskip

\begin{Remark}\label{rem23}
In Theorem \ref{them1}, for each $j$, $\|\wt U_j-U_j\|_F$ measures the perturbation of the eigenspace $\mathcal R(U_j)$. Therefore,
the inequality (\ref{c-2-4}) actually bounds the perturbations of the eigenspaces $\mathcal R(U_1),\ldots,\mathcal R(U_k)$ and their corresponding eigenvalues.
The combined bound (\ref{c-13}) contains  perturbations of one eigenspace  and its corresponding eigenvalues.  Following the same notations given in Theorem \ref{them1} and applying (\ref{c-13}) to the eigenspaces $\mathcal{R}(U_1),\ldots,\mathcal{R}(U_k)$, we can get

\begin{eqnarray} \label{totalb}
\|{\rm Eig}^{\downarrow}(\wt A)-{\rm Eig}^{\downarrow}(A)\|_F^2
+\sum_{j=1}^k{\wt\delta}^2_{j}\,\|\sin\Theta(U_j,\wt U_j)\|_2^2
\le \|\Delta A\|_F^2,
\end{eqnarray}
where
\[
\wt \delta_{j}^2=:\delta_{j}^2-\|{\rm Eig}^{\downarrow}(\Lambda_j(1))-{\rm Eig}^{\downarrow}(\Lambda_j(0))\|_F^2,\qquad
\delta_{j}=\min_{\lambda_1\in\lambda(\Lambda_j)
, \lambda_2\in\cup_{i\ne j}\lambda(\wt\Lambda_i)}\{|\lambda_1-\lambda_2|\},
\]
for $j=1,2,\ldots,k$.
Since (\cite[Chapter I, Theorem 5.5]{ss90})
\[\|\sin\Theta(U_j,\wt U_j)\|_F^2=\sum_{i=1,i\ne j}^k\|U_i^H\wt U_j\|_F^2
\]
and
\begin{eqnarray}
\label{esin}
\|\wt U_j-U_j\|_F^2&=&\|U^H(\wt U_j-U_j)\|_F^2
=\|U_j^H\wt U_j-I\|_F^2+\sum_{i=1,i\ne j}^k\|U_i^H\wt U_j\|_F^2 \nonumber\\
&\ge&\sum_{i=1,i\ne j}^k\|U_i^H\wt U_j\|_F^2=\|\sin\Theta(U_j,\wt U_j)\|_F^2,
\end{eqnarray}
when $\delta_{j,\min}^2$ is sufficiently close to  $\wt \delta_{j}^2$ for all $j$,
(\ref{c-2-4}) implies (\ref{totalb}). Note when $\delta_{j,\min}>0$, one can verify that
$\delta_{j,\min}-\wt\delta _j=O(\|\Delta A\|_F)$ when $\|\Delta A\|_F$ is sufficiently small.
\end{Remark}

\medskip

The following result gives a combined perturbation bound for one eigenspace and its corresponding  eigenvalues of a Hermitian matrix, which is similar to (\ref{c-13}). Without loss of generality, we consider
$\mathcal R(U_1)$, where $U_1$ is defined in Theorem~\ref{them1}.

\medskip
\begin{Theorem}\label{them2}
Under the assumptions of Theorem \ref{them1}, if $\|\Delta A\|_2<\delta_{1,\min}$,  then we have
{\footnotesize\begin{eqnarray}\label{c-2-15}
\left(1-\frac{\|{\Delta A}\|_2}{\delta_{1,\min}}\right)^2\|{\rm Eig}^{\downarrow}(\wt \Lambda_1)-{\rm Eig}^{\downarrow}(\Lambda_1)\|_F^2+\left(\delta_{1,\min}-\|{\Delta A}\|_2\right)^2\|\wt U_1-U_1\|_F^2\leq \|\Delta AU_1\|_F^2,
\end{eqnarray}}
where $\wt \Lambda_1=\Lambda_1(1)$ and $\Lambda_1=\Lambda_1(0)$.
\end{Theorem}

\medskip
\begin{proof} From (\ref{c-2-3}), (\ref{c-2-5}) and (\ref{c-2-8}),
it is easily seen that
\begin{eqnarray}
\nonumber
\|{\Delta AU_1(t)}\|_F^2&=&
\|{U(t)^H\Delta AU_1(t)}\|_F^2=\|{\Delta A_{11}(t)}\|_F^2+\sum_{i=2}^k\|{\Delta A_{i1}(t)}\|_F^2\\
\nonumber
&\geq&\|{\dot\Lambda_1(t)}\|_F^2+\sum_{i=2}^k\delta_{1i}(t)^2\|{U_i(t)^H\dot U_1(t)}\|_F^2\nonumber\\
&\ge &\|{\dot\Lambda_1(t)}\|_F^2
+\delta_1(t)^2\sum_{i=2}^k\|{U_i(t)^H\dot U_1(t)}\|_F^2\nonumber\\
&\ge& \|{\dot\Lambda_1(t)}\|_F^2+\delta_1(t)^2\|{U(t)^H\dot U_1(t)}\|_F^2
\nonumber\\
\label{1stc}
&=&\|{\dot\Lambda_1(t)}\|_F^2+\delta_1(t)^2\|{\dot U_1(t)}\|_F^2.
\end{eqnarray}
By taking integrals on both side of (\ref{1stc}) on the interval $[0,\,1]$, as before
one can derive
\begin{eqnarray}\label{c-2-16}
\|\wt\Lambda_1-\Lambda_1\|_F^2+\delta_{1,\min}^2\|{\wt U_1-U_1}\|_F^2
\leq\int_0^1\|{\Delta A U_1(t)}\|_F^2{\rm d}t.
\end{eqnarray}
Noting that $U_1=U_1(0)$, we have
\begin{eqnarray}\label{c-2-17}
\|{\Delta A U_1(t)}\|_F&=&
\|{\Delta AU_1+\Delta A(U_1(t)-U_1(0))}\|_F\le
 \|{\Delta AU_1}\|_F+\|{\Delta A}\|_2\|{U_1(t)-U_1(0)}\|_F\nonumber\\
&\le& \|{\Delta AU_1}\|_F+\|{\Delta A}\|_2\|{U_1(t^\star)-U_1(0)}\|_F,
\end{eqnarray}
where $t^\star\in [0,1]$ satisfies
\[
\|{U_1(t^\star)-U_1(0)}\|_F=\max_{0\le t\le 1}\|{U_1(t)-U_1(0)}\|_F,
\]
\noindent and from (\ref{1stc}),
\bstar
\delta_{1,\min}^2\|{U_1(t^\star)-U_1(0)}\|_F^2
&\le&\left(\min_{0\le t\le t^\star}\delta_1(t)^2\right)\left\|\int_0^{t^\star}\dot U_1(t){\rm d}t\right\|_F^2
\leq\int_0^{t^\star}\delta_1(t)^2\|{\dot U_1(t)}\|_F^2{\rm d}t\\
&\le&\int_0^{t^\star}\|{\Delta A U_1(t)}\|_F^2{\rm d}t\\
&\le&\int_0^{t^\star} \left(\|{\Delta AU_1}\|_F+\|{\Delta A}\|_2\|{U_1(t^\star)-U_1(0)}\|_F\right)^2{\rm d}t\\
&=&t^\ast  \left(\|{\Delta AU_1}\|_F+\|{\Delta A}\|_2\|{U_1(t^\star)-U_1(0)}\|_F\right)^2\\
&\leq&  \left(\|{\Delta AU_1}\|_F+\|{\Delta A}\|_2\|{U_1(t^\star)-U_1(0)}\|_F\right)^2,
\estar
which leads to
\begin{eqnarray}\label{c-2-18}
\|{U_1(t^\star)-U_1(0)}\|_F\leq \frac1 {\delta_{1,\min}-\|{\Delta A}\|_2}\|{\Delta AU_1}\|_F.
\end{eqnarray}
Hence by (\ref{c-2-17}) and (\ref{c-2-18}) we get
\begin{eqnarray}\label{c-2-19}
\|{\Delta A U_1(t)}\|_F\leq
\frac{\delta_{1,\min}}{\delta_{1,\min}-\|{\Delta A}\|_2}\|{\Delta AU_1}\|_F.
\end{eqnarray}
Then the bound (\ref{c-2-15}) follows from (\ref{c-2-16}), (\ref{c-2-19}) and the fact
\begin{eqnarray*}
&&\|{\rm Eig}^{\downarrow}(\wt\Lambda_1)-{\rm Eig}^{\downarrow}(\Lambda_1)\|_F
\leq \| \wt\Lambda_1-\Lambda_1\|_F.
\end{eqnarray*}
The proof is complete.
\end{proof}

\medskip

\begin{Corollary}\label{cor1}
Under the assumptions of Theorem \ref{them2}, we have
{\footnotesize \begin{eqnarray}\label{c-2-20}
\left(1-\frac{\|{\Delta A}\|_2}{\delta_{1,\min}}\right)^2\|{\rm Eig}^{\downarrow}(\wt \Lambda_1)-{\rm Eig}^{\downarrow}(\Lambda_1)\|_F^2+\left(\delta_{1,\min}-\|{\Delta A}\|_2\right)^2\|\sin\Theta(U_1,\wt U_1)\|_F^2\leq \|\Delta AU_1\|_F^2.
\end{eqnarray}}
In particular,
{\footnotesize\begin{eqnarray}\label{c-2-21}
\|\sin\Theta(U_1,\wt U_1)\|_F\leq \frac{1}{\delta_{1,\min}-\|{\Delta A}\|_2}\|\Delta AU_1\|_F.
\end{eqnarray}}
\end{Corollary}

\medskip

\begin{proof}
The bound (\ref{c-2-20}) is from (\ref{c-2-15}) and (\ref{esin}) (with $j=1$),
and (\ref{c-2-21}) follows from (\ref{c-2-20})
by dropping the eigenvalue error term.
\end{proof}

\medskip

\begin{Remark}
The inequalities (\ref{c-2-20}) and (\ref{c-2-21}) are similar to  (\ref{c-13}) and (\ref{c-122}), but
they require $\|\Delta A\|_2<\delta_{1,\min}$. In this sense the bounds (\ref{c-2-20}) and (\ref{c-2-21})
as well as (\ref{c-2-15}) are local. Therefore, it is not simple to compare these bounds with
 (\ref{c-13}) and (\ref{c-122}).  Following the discussions in Remark~\ref{rem24},
 a sufficient condition for $\|\Delta A\|_2<\delta_{1,\min}$ is  $\|\Delta A\|_2<\delta_1(0)/3.$
\end{Remark}

\medskip

\begin{Remark} Applying the Mean Value Theorem to the integral in (\ref{c-2-16}), we
have a simpler bound
\[
\|{\rm Eig}^{\downarrow}(\wt\Lambda_1)-{\rm Eig}^{\downarrow}(\Lambda_1)\|_F^2+\delta_{1,\min}^2\|{\wt U_1-U_1}\|_F^2
\leq\|{\Delta A U_1(t_0)}\|_F^2,
\]
for some $t_0\in [0,\,1]$.
\end{Remark}

\section{Combined bounds of singular values and singular subspaces}
\label{sec3} In this section we will derive combined perturbation bounds for singular subspaces and corresponding singular values of a general matrix.
The following bound will be needed for derivations.
\vskip 0.1in
\begin{Lemma} (\cite[Chapter IV, Theorem 4.11]{ss90})\label{lem31}
 Let $\wt B=B+\Delta B\in\mathbb{C}^{m\times n}$. Then
\begin{eqnarray}\label{c-3-0}
\|{\rm Sing^{\downarrow}}(\wt B)-{\rm Sing}^{\downarrow}(B)\|_F\leq \|\Delta B\|_F.
\end{eqnarray}
\end{Lemma}

\medskip
Let $B,\Delta B\in\mathbb{C}^{m\times n}$ with $m\ge n$.
For any $t\in\mathbb{R}$, the matrix $B(t)=B+t\Delta B$ has an analytic factorization  (\cite{bunbmn91,de99})
\begin{eqnarray}
\label{gsvd}
B(t)=W(t)\mat{c}\Sigma(t)\\0\rix V(t)^H,\quad \Sigma(t)=\diag(\Sigma_1(t),\ldots,\Sigma_k(t)),
\end{eqnarray}
where $\Sigma_j(t)\in{\mathbb C}^{r_j\times r_j}$ for $j=1,\ldots,k$,
 are analytic but not necessarily diagonal, $r_1+\cdots+r_k=n$,
\begin{eqnarray}\label{bwv}
 W(t)=\left[ W_1(t),\ldots, W_k(t), W_{k+1}(t)\right]\in {\mathbb C}^{m\times m},\qquad V(t)=\left[V_1(t), \ldots, V_k(t)\right]\in {\mathbb C}^{n\times n}
\end{eqnarray}
are unitary and analytic
and $W_j(t)\in{\mathbb C}^{m\times r_j}, V_j(t)\in{\mathbb C}^{n\times r_j}$ for
$j=1,\ldots,k$, $W_{k+1}(t)\in{\mathbb C}^{m\times (m-n)}$.
The analytic factorization (\ref{gsvd}) is not unique. Similar to Lemma \ref{lem22}, we have the following results with special choices of $W(t)$ and $V(t)$.

\medskip

\begin{Lemma}\label{lem32}
Let $B(t)=B+t\Delta B\in\mathbb{C}^{m\times n}$ with $m\geq n$ and $t\in\mathbb{R}$.
Then $B(t)$ has the analytic decomposition (\ref{gsvd}) and (\ref{bwv}), and $W(t)$, $V(t)$
satisfy
\begin{eqnarray}\label{c-3-3}
W_j^H(t)\dot{W}_j(t)=0,\quad j=1,\ldots,k+1,\quad\mbox{and}\quad V_j(t)^H\dot{V}_j(t)=0,
\quad j=1,\ldots,k.
\end{eqnarray}
\end{Lemma}

\medskip

\begin{proof} Let $B(t)$ have an analytic decomposition (\ref{gsvd}) and (\ref{bwv}).
Then for any block diagonal unitary analytic matrices
\[
P(t)=\diag(P_1(t),\dots, P_k(t),P_{k+1}(t)),\quad
Q(t)=\diag(Q_1(t),\ldots,Q_k(t))
\]
with $P_j(t),Q_j(t)\in{\mathbb C}^{r_j\times r_j}$ for $j=1,\ldots,k$, and
$P_{k+1}(t)\in{\mathbb C}^{(m-n)\times (m-n)}$,
\[
B(t)=(W(t)P(t))\mat{ccc}P_1^H\Sigma_1(t)Q_1(t)&\ldots&0\\\vdots&\ddots&\vdots\\
0&\ldots&P_k(t)^H\Sigma_k(t)Q_k(t)\\0&\ldots&0\rix (V(t)Q(t))^H,
\]
is a factorization in the same form of (\ref{gsvd}).  Following
Lemma \ref{lem21}, we can show that there exist $P(t)$ and $Q(t)$ such that
for the new $W(t):=W(t)P(t)$ and $V(t):=V(t)Q(t)$ the conditions in (\ref{c-3-3}) are satisfied and
$B(t)$ still has an analytic decomposition (\ref{gsvd}) with the new
block diagonal matarix $\Sigma(t):=\diag(P_1(t),\ldots,P_k(t))^H\Sigma(t)Q(t)$.
\end{proof}

\medskip

\begin{Theorem}\label{them31}
Suppose that $B(t)=B+t\Delta B\in\mathbb{C}^{m\times n}$ with $
m\geq n$ and $t\in\mathbb{R}$ has the analytic decomposition (\ref{gsvd}) and (\ref{bwv})
with $W(t)$ and $V(t)$ satisfying (\ref{c-3-3}). Define
\begin{eqnarray*}
&&\rho_{ji}(t)=\min_{\sigma_1(t)\in\sigma(\Sigma_j(t)),\sigma_2(t)\in\sigma(\Sigma_i(t))}\{\left|\sigma_1(t)-\sigma_2(t)\right|\}=\rho_{ij}(t),
\qquad \rho_{j}(t)=\min_{i\ne j}\{\rho_{ji}(t)\},\\
&&
\sigma_{j,\min}(t)=\min \{\sigma(\Sigma_j(t))\},\quad\wh \rho_{j}(t)=\min\left\{\rho_{j}(t),\sigma_{j,\min}(t)\right\},\quad
\sigma_{\min}(t)=\min_{j}\{\sigma_{j,\min}(t)\}
\end{eqnarray*}
and
\begin{eqnarray*}
\rho_{j,\min}=\min_{0\le t\le 1}\rho_j(t),\qquad\wh \rho_{j,\min}=
\min_{0\le t\le 1}\wh\rho_j(t),\quad
\sigma_{\min}=\min_{0\le t\le 1}\sigma_{\min}(t).
\end{eqnarray*}
Let
\bstar
W(0)&=&\left[W_1(0),\ldots, W_{k+1}(0)\right]=:\left[W_1,\ldots, W_{k+1}\right]=W,\\
W(1)&=&\left[W_1(1),\ldots,W_{k+1}(1)\right]=:[\wt W_1,\ldots,\wt W_{k+1}]
=\wt W,\\
V(0)&=&[V_1(0),\ldots, V_{k}(0)]=:[ V_1, \ldots,V_{k}]=V,\\
V(1)&=&[V_1(1),\ldots ,V_{k}(1)]=:[\wt V_1,\ldots,\wt V_{k}]=\wt V,
\estar
and $B(1)=:\wt B$. Then
\begin{eqnarray}\label{c-3-4}
&&\|{{\rm Sing}^{\downarrow}(\wt B)-{\rm Sing}^{\downarrow}(B)}\|_F^2
+\sum_{j=1}^{k}\frac{\wh\rho_{j,\min}^2}2\|{\wt W_j-W_j}\|_F^2
+\frac{\sigma_{\min}^2}2\|{\wt W_{k+1}-W_{k+1}}\|_F^2\nonumber\\
&&\qquad\qquad\qquad\qquad\qquad\qquad\qquad\qquad\quad
+\sum_{j=1}^k\frac{\rho_{j,\min}^2}2\|{\wt V_j-V_j}\|_F^2
\leq\|{\Delta B}\|_F^2.
\end{eqnarray}
In particular, when $m=n$ we have
\begin{eqnarray}\label{c-3-4-1}
\|{{\rm Sing}^{\downarrow}(\wt B)-{\rm Sing}^{\downarrow}(B)}\|_F^2
+\sum_{j=1}^k\frac{\rho_{j,\min}^2}2\left(\|{\wt W_j-W_j}\|_F^2+\|{\wt V_j-V_j}\|_F^2\right)
\leq\|{\Delta B}\|_F^2.
\end{eqnarray}
\end{Theorem}

\medskip
\begin{proof}
By taking the derivation on both sides of (\ref{gsvd}), $W(t)^HW(t)=I$ and $V(t)^HV(t)=I$,
respectively, we obtain
\begin{eqnarray}\label{c-3-5}
\Delta B =\dot{W}(t)\mat{c}\Sigma(t)\\0\rix V(t)^H+W(t)\mat{c}\dot{\Sigma}(t)\\0\rix V(t)^H
+W(t)\mat{c}\Sigma(t)\\0\rix \dot{V}(t)^H
\end{eqnarray}
and
\begin{eqnarray}\label{c-3-6}
\dot W(t)^H W(t)=-W(t)^H\dot W(t),\qquad \dot V(t)^H V(t)=-V(t)^H\dot V(t).
\end{eqnarray}
 Using the second equality of (\ref{c-3-6}), one can rewrite (\ref{c-3-5}) as
\begin{eqnarray}\label{dsvd-1}\mbox{}\quad
W(t)^H\Delta BV(t) =W(t)^H\dot{W}(t)\mat{c}\Sigma(t)\\0\rix-\mat{c}\Sigma(t)\\0\rix V(t)^H\dot{V}(t)
+\mat{c}\dot{\Sigma}(t)\\0\rix.
\end{eqnarray}
Partition
$$
W(t)^H\dot{W}(t)=[W_{ij}(t)],\quad V(t)^H\dot{V}(t)=[V_{ij}(t)],\quad
 W(t)^H\Delta B V(t)=[\Delta B_{ij}(t)],
$$
where $W_{ij}(t)=W_i(t)^H\dot{W}_j(t)$, $V_{ij}(t)=V_i(t)^H\dot{V}_j(t)$, and
$\Delta B_{ij}(t)=W_i(t)^H\Delta BV_j(t)$.
From (\ref{c-3-3}) and (\ref{c-3-6}), we have
\begin{eqnarray}\label{skew}
W_{ij}(t)=-W_{ji}(t)^H,\quad
V_{ij}(t)=-V_{ji}(t)^H,\qquad \forall i\ne j
\end{eqnarray}
and
\begin{eqnarray}\label{skewd}
W_{ii}(t)=0,\qquad V_{ii}(t)=0,\quad \forall i.
\end{eqnarray}
Then (\ref{dsvd-1}) implies
\begin{eqnarray}\label{diag}
\dot{\Sigma}_j(t)=\Delta B_{jj}(t),\quad j=1,\ldots,k,
\end{eqnarray}
\begin{eqnarray}\label{dsvd-3}
\Sigma_i(t)^HW_{ij}(t)-V_{ij}(t)\Sigma_j(t)^H=-\Delta B_{ji}(t)^H,
\quad
W_{ij}(t)\Sigma_j(t)-\Sigma_i(t)V_{ij}(t)=\Delta B_{ij}(t)
\end{eqnarray}
for $1\le j<i\le k$, and
\begin{eqnarray}\label{w2}
W_{k+1,j}(t)\Sigma_j(t)=\Delta B_{k+1,j}(t),\quad  j=1,\ldots,k.
\end{eqnarray}
For each $j=1,\ldots,k$, let
\begin{eqnarray*}
\Sigma_j(t)=G_j(t)\wh\Sigma_j(t)F_j(t)^H
\end{eqnarray*}
be an SVD (not necessarily analytic) of $\Sigma_j(t)$, where
\begin{eqnarray*}
\wh \Sigma_j(t)=\diag(\sigma_{j,1}(t),\ldots,\sigma_{j,r_j}(t)),\qquad \sigma_{j,q}(t)\geq 0, \quad q=1,\ldots,r_j.
\end{eqnarray*}
Denote
\begin{eqnarray*}
&&\wh W_{ij}(t)=G_i(t)^HW_{ij}(t)G_j(t)=[w_{pq}^{(ij)}(t)],\quad
\wh V_{ij}(t)=F_i(t)^HV_{ij}(t)F_j(t)=[v_{pq}^{(ij)}(t)],\\
&&\wh B_{ji}(t)=G_j(t)^H\Delta B_{ji}(t)F_i(t)=[b_{pq}^{(ji)}(t)],\quad
\wh B_{ij}(t)=G_i(t)^H\Delta B_{ij}(t)F_j(t)=[c_{pq}^{(ij)}(t)],
\end{eqnarray*}
for $1\le j<i\le k$, and
\begin{eqnarray*}
&&\wh W_{k+1,j}(t)=W_{k+1,j}(t)G_j(t)=[x_{pq}^{(j)}(t)],\quad
\wh B_{k+1,j}(t)=\Delta B_{k+1,j}(t)F_j(t)=[f_{pq}^{(j)}(t)],
\end{eqnarray*}
for $j=1,\ldots,k$. From (\ref{dsvd-3}) and (\ref{w2}), one has
\[
\wh\Sigma_i(t)\wh W_{ij}(t)-\wh V_{ij}(t)\wh \Sigma_j(t)=-\wh B_{ji}(t)^H,\quad
\wh W_{ij}(t)\wh\Sigma_j(t)-\wh \Sigma_i(t)\wh V_{ij}(t)=\wh B_{ij}(t),
\]
for $1\le j<i\le k$  and
\[
\wh W_{k+1,j}(t)\wh\Sigma_j(t)=\wh B_{k+1,j}(t),\quad j=1,\ldots,k,
\]
which imply
\begin{eqnarray}\label{c-3-7}
-\overline{b_{qp}^{(ji)}(t)} = \sigma_{i,p}(t)w_{pq}^{(ij)}(t)-\sigma_{j,q}(t)v_{pq}^{(ij)}(t),\quad
c_{pq}^{(ij)}(t)=\sigma_{j,q}(t)w_{pq}^{(ij)}(t)-\sigma_{i,p}(t)v_{pq}^{(ij)}(t),
\end{eqnarray}
for  $p=1,\ldots,r_i,q=1,\ldots,r_j$, $1\le j<i\le k$,
 and
\begin{eqnarray*}
f_{pq}^{(j)}(t) = x_{pq}^{(j)}(t)\sigma_{j,q}(t),\quad p=1,\ldots,m-n,\quad
q=1,\ldots,r_j,\quad j=1,\ldots,k,
\end{eqnarray*}
and from which
\begin{eqnarray}\label{c-3-8}
\|\Delta B_{k+1,j}(t)\|_F^2\ge
\sigma_{j,\min}(t)^2\|W_{k+1,j}(t)\|_F^2.
\end{eqnarray}
By (\ref{c-3-7}), simple calculations yield
\begin{eqnarray*}
|b_{qp}^{(ji)}(t)|^2+|c_{pq}^{(ij)}(t)|^2
\geq (\sigma_{i,p}(t)-\sigma_{j,q}(t))^2(|w_{pq}^{(ij)}(t)|^2+|v_{pq}^{(ij)}(t)|^2),
\end{eqnarray*}
which implies
\begin{eqnarray}\label{c-3-9}
\|\Delta B_{ji}(t)\|_F^2+\|\Delta B_{ij}(t)\|_F^2\ge \rho_{ji}(t)^2(\|W_{ij}(t)\|_F^2+\|V_{ij}(t)\|_F^2),
\quad 1\le j<i\le k.
\end{eqnarray}
Then using (\ref{diag}), (\ref{c-3-8}), (\ref{c-3-9})
and $\rho_{ij}(t)=\rho_{ji}(t)$, we obtain
{\footnotesize\begin{eqnarray}\label{c-3-10}
2\|{\Delta B}\|_F^2&=&2\|{W(t)^H\Delta BV(t)}\|_F^2\nonumber\\
&=&
2\sum_{j=1}^k \|{\Delta B_{jj}(t)}\|_F^2+2\sum_{1\le j<i\le k}(\|{\Delta B_{ji}(t)}\|_F^2
+\|{\Delta B_{ij}(t)}\|_F^2)
+2\sum_{j=1}^k\|{\Delta B_{k+1,j}(t)}\|_F^2\nonumber\\
&\ge&2\sum_{j=1}^k\|{\dot\Sigma_j(t)}\|_F^2
+2\sum_{1\le j<i\le k}\rho_{ji}(t)^2(\|{W_{ij}(t)}\|_F^2+\|{V_{ij}(t)}\|_F^2)
+2\sum_{j=1}^{k}\sigma_{j,\min}(t)^2\|W_{k+1,j}(t)\|_F^2\nonumber\\
&=&2\|{\dot\Sigma(t)}\|_F^2+\sum_{j=1}^k\sum_{i=1}^k\rho_{ji}(t)^2(\|{W_{ij}(t)}\|_F^2+\|{V_{ij}(t)}\|_F^2)
+\sum_{j=1}^{k}\sigma_{j,\min}(t)^2(\|W_{k+1,j}(t)\|_F^2+\|W_{j,k+1}(t)\|_F^2)\nonumber\\
&=&2\|{\dot\Sigma(t)}\|_F^2
+\sum_{j=1}^k\left(\sum_{i=1}^k\rho_{ji}(t)^2\|{W_{ij}(t)}\|_F^2
+\sigma_{j,\min}(t)^2\|W_{k+1,j}(t)\|_F^2\right)\nonumber\\
&&\qquad\qquad\qquad+\sum_{j=1}^k\sum_{i=1}^k\rho_{ji}(t)^2\|{V_{ij}(t)}\|_F^2
+\sum_{j=1}^{k}\sigma_{j,\min}(t)^2\|W_{j,k+1}(t)\|_F^2\nonumber\\
&\ge&2\|{\dot\Sigma(t)}\|_F^2
+\sum_{j=1}^k\wh \rho_j(t)^2\|{\dot W_{j}(t)}\|_F^2+\sum_{j=1}^k\rho_j(t)^2
\|{\dot V_{j}(t)}\|_F^2+\sigma_{\min}(t)^2\|{\dot W_{k+1}(t)}\|_F^2\nonumber\\
&\ge&2\|{\dot\Sigma(t)}\|_F^2
+\sum_{j=1}^k\wh \rho_{j,\min}^2\|{\dot W_{j}(t)}\|_F^2+\sum_{j=1}^k\rho_{j,\min}^2
\|{\dot V_{j}(t)}\|_F^2+\sigma_{\min}^2\|{\dot W_{k+1}(t)}\|_F^2.
\end{eqnarray}}
Then taking the integral on both sides of (\ref{c-3-10}) on the internal $[0,~ 1]$ yields
\begin{eqnarray}\label{c-3-12}
2\|{\Delta B}\|_F^2&=&
\int_0^12\|{\Delta B}\|_F^2{\rm d}t \ge 2\int_0^1\|{\dot\Sigma(t)}\|_F^2{\rm d}t
+\sum_{j=1}^k\wh\rho_{j,\min}^2\int_0^1\|{\dot W_j(t)}\|_F^2{\rm d}t\nonumber\\
&&\qquad\qquad+\sigma_{\min}^2\int_0^1\|{\dot W_{k+1}(t)}\|_F^2{\rm d}t+
\sum_{j=1}^k\rho_{j,\min}^2\int_0^1\|{\dot V(t)}\|_F^2
{\rm d}t\nonumber\\
&\ge&2\|\Sigma(1)-\Sigma(0)\|_F^2
+\sum_{j=1}^k\wh\rho_{j,\min}^2\|{\wt W_j-W_j}\|_F^2+
\sigma_{\min}^2\|{\wt W_{k+1}-W_{k+1}}\|_F^2\nonumber\\
&&\qquad\qquad\qquad\qquad\qquad\qquad
+ \sum_{j=1}^k\rho_{j,\min}^2\|{\wt V_j-V_j}\|_F^2.
\end{eqnarray}
Since $\sigma(B)=\sigma(\Sigma(0))$ and $\sigma(\wt B)=\sigma(\Sigma(1))$, it follows from (\ref{c-3-0}) that
\begin{eqnarray*}
\|{\rm Sing}^{\downarrow}(\wt B)-{\rm Sing}^{\downarrow}(B)\|_F
=\|{\rm Sing}^{\downarrow}(\Sigma(1))-{\rm Sing}^{\downarrow}(\Sigma(0))\|_F
\leq \|\Sigma(1)-\Sigma(0)\|_F.
\end{eqnarray*}
Combining it with (\ref{c-3-12}) 
leads to (\ref{c-3-4}).

\medskip
When $m=n$, $W_{k+1,1}(t),\ldots,W_{k+1,k}(t)$ are void and (\ref{c-3-4-1}) is derived in the
same way.
\end{proof}

\medskip

\begin{Remark}\label{lem311}
The inequalities (\ref{c-3-4}) and (\ref{c-3-4-1}) bound the perturbations
of all the left and right singular subspaces $\mathcal R(W_j),\mathcal R(V_j)$,
$j=1,\ldots,k$, and the nullspace $\mathcal R(W_{k+1})$ of $B^H$ as well as
all the singular values.  Obviously, the bound (\ref{c-3-4}) is sharper than the one (\ref{c-3-0}).
\end{Remark}\vskip 2mm

\begin{Remark}\label{lem321}
When $k=n$ and $r_1=\cdots=r_n=1$, (\ref{gsvd}) may not necessarily an SVD, since
$\Sigma_1(t),\ldots,\Sigma_n(t)$ (which are scalars now) are not necessarily nonnegative.
Consequently, $W_1,\ldots,W_n$ and $V_1,\ldots,V_n$ are not necessarily
 left and right singular vectors of $B$. However, we may use a constant diagonal unitary
 matrix $P$ to rewrite (\ref{gsvd}) as $B(t)=W(t)\mat{c}\Sigma(t)P\\0\rix (V(t)P)^H$
 such that the diagonal entries of $\Sigma(0)P$ are nonnegative, i.e., they are the
 singular values of $B$. Then the columns of $W(0)$ and $V(0)P$ are the left and right
 singular vectors. Without loss of generality, we set $P=I$, then when $m>n$,
  (\ref{c-3-4}) becomes
 {\footnotesize\[
\|{{\rm Sing}^{\downarrow}(\wt B)-{\rm Sing}^{\downarrow}(B)}\|_F^2
+\sum_{j=1}^{n}\frac{\wh\rho_{j,\min}^2}2\|{\wt W_j-W_j}\|_F^2
+\frac{\sigma_{\min}^2}2\|{\wt W_{n+1}-W_{n+1}}\|_F^2
+\sum_{j=1}^n\frac{\rho_{j,\min}^2}2\|{\wt V_j-V_j}\|_F^2
\leq\|{\Delta B}\|_F^2.
\]}
 and when $m=n$, (\ref{c-3-4-1}) becomes
 \[
 \|{{\rm Sing}^{\downarrow}(\wt B)-{\rm Sing}^{\downarrow}(B)}\|_F^2
+\sum_{j=1}^{n}\frac{\rho_{j,\min}^2}2\left(\|{\wt W_j-W_j}\|_F^2+\|{\wt V_j-V_j}\|_F^2\right)
\leq\|{\Delta B}\|_F^2,
\]
and from which we have
 \[
\|{{\rm Sing}^{\downarrow}(\wt B)-{\rm Sing}^{\downarrow}(B)}\|_F^2
+\frac{\wh\rho_{\min}^2}2\|{\wt W-W}\|_F^2
+\frac{\rho_{\min}^2}2\|{\wt V-V}\|_F^2
\leq\|{\Delta B}\|_F^2,\qquad \mbox{when }m>n,
\]
 \[
 \|{{\rm Sing}^{\downarrow}(\wt B)-{\rm Sing}^{\downarrow}(B)}\|_F^2
+\frac{\rho_{\min}^2}2\left(\|{\wt W-W}\|_F^2+\|{\wt V-V}\|_F^2\right)
\leq\|{\Delta B}\|_F^2,\qquad \mbox{when }m=n,
\]
where $\wh\rho_{\min}=\min\{\sigma_{\min},\min_j\{\wh\rho_{j,\min}\}\}$ and
$\rho_{\min}=\min_j\{\rho_{j,\min}\}$.

If $k=1$ and $m=n$, (\ref{c-3-3}) implies that $\wt W=W$ and $\wt V=V$. In this case, (\ref{c-3-4-1}) reduces to (\ref{c-3-0}). If $k=1$ and $m>n$, (\ref{c-3-4}) is still sharper
than (\ref{c-3-0}), because there is one additional term $\frac{\sigma_{\min}^2}2\|\wt W_2-W_2\|_F^2$
on the left hand side of (\ref{c-3-4}).
\end{Remark}

\medskip

\begin{Remark}\label{rem331}
Similar to Remark \ref{rem23} in the previous section, the combined perturbation bounds (\ref{c-3-4}) and (\ref{c-3-4-1}) can be compared
with the corresponding results in \cite{li07}.
\end{Remark}\vskip 2mm

\begin{Remark}Using Theorem 4.13 in \cite[Chapter IV]{ss90}. we can obtain a sufficient condition for $\wh \rho_{j,min}>0$:
$$
\|\Delta B\|_2\leq \min\{\rho_j(0)/2,~\sigma_{j,\min}(0)\},
$$
where $\rho_j(0)$ is defined in Theorem \ref{them31}. Its proof is similar to the one in Remark \ref{rem24}.
\end{Remark}

\medskip

The following result gives a combined perturbation bound of a pair of left and right singular subspaces and the corresponding singular values.

\medskip

\begin{Theorem}\label{them32}
Under the assumptions of Theorem \ref{them31}, if $\|\Delta B\|_2<\wh \rho_{1,\min}$,  then
{\footnotesize \begin{eqnarray}\label{c-3-14-2}
&&2\|{{\rm Sing}^{\downarrow}(\wt\Sigma_1)-{\rm Sing}^{\downarrow}(\Sigma_1)}\|_F^2+\wh \rho_{1,\min}^2\|{\wt W_1-W_1}\|_F^2+
\rho_{1,\min}^2\|{\wt V_1-V_1}\|_F^2\nonumber\\
&\le&
\left(\frac{\wh \rho_{1,\min}}{\wh \rho_{1,\min}-\|{\Delta B}\|_2}\right)^2(\|{\Delta BV_1}\|_F^2+\|{W_1^H\Delta B}\|_F^2),
\end{eqnarray}}
where $\wt \Sigma_1=:\Sigma_1(1)$ and $\Sigma_1=:\Sigma_1(0)$.
When $m=n$, if  $\|\Delta B\|_2<\rho_{1,\min}$, we have
{\footnotesize\begin{eqnarray}\label{c-3-14-3}
&&2\left(1-\frac{\|\Delta B\|_2}{\rho_{1,\min}}\right)^2
\|{{\rm Sing}^{\downarrow}(\wt \Sigma_1)-{\rm Sing}^{\downarrow}(\Sigma_1)}\|_F^2+ (\rho_{1,\min}-\|\Delta B\|_2)^2\left(\|{\wt W_1- W_1}\|_F^2+
\|{\wt V_1-V_1}\|_F^2\right)\nonumber\\
& \le& \|{\Delta B V_1}\|_F^2+\|{W_1^H\Delta B}\|_F^2.
\end{eqnarray}}
\end{Theorem}

\medskip

\begin{proof} Similar to the case for deriving (\ref{c-3-10}), we have
\begin{eqnarray}
\nonumber
&&\|{\Delta BV_1(t)}\|_F^2+
\|{W_1(t)^H\Delta B}\|_F^2
=\|{W(t)^H\Delta BV_1(t)}\|_F^2+
\|{W_1(t)^H\Delta BV(t)}\|_F^2\nonumber\\
\nonumber
&&\qquad=2\|{\Delta B_{11}(t)}\|_F^2
+\sum_{i=2}^k(\|{\Delta B_{i1}(t)}\|_F^2+\|{\Delta B_{1i}(t)}\|_F^2)
+\|{\Delta B_{k+1,1}(t)}\|_F^2\\
\nonumber
&&\qquad\ge2\|{\dot\Sigma_1(t)}\|_F^2+
\sum_{i=2}^k\rho_{1i}(t)^2(\|{W_{i1}(t)}\|_F^2+\|{V_{i1}(t)}\|_F^2)
+\sigma_{1,\min}(t)^2\|{W_{k+1,1}(t)}\|_F^2\\
\nonumber
&&\qquad\ge2\|{\dot\Sigma_1(t)}\|_F^2+\wh\rho_{1,\min}^2(\|{W(t)^H\dot W_{1}(t)}\|_F^2
+\rho_{1,\min}^2\|{V(t)^H\dot V_{1}(t)}\|_F^2\\
\label{ldtbd}
&&\qquad\ge 2\|{\dot\Sigma_1(t)}\|_F^2+\wh\rho_{1,\min}^2\|{\dot W_{1}(t)}\|_F^2+
\rho_{1,\min}^2\|{\dot V_{1}(t)}\|_F^2.
\end{eqnarray}
For any $t\in [0,~1]$ we have
\begin{eqnarray}\label{c-3-15}
\|{\Delta BV_1(t)}\|_F&=&\|{\Delta BV_1+\Delta B(V_1(t)-V_1(0))}\|_F\nonumber\\
&\le& \|{\Delta BV_1}\|_F+\|{\Delta B}\|_2\|{V_1(t)-V_1(0)}\|_F\nonumber\\
&\le& \|{\Delta BV_1}\|_F+\|{\Delta B}\|_2\|{V_1(t^\star)-V_1(0)}\|_F
\end{eqnarray}
and
\begin{eqnarray}\label{c-3-16}
\|{W_1(t)^H\Delta B}\|_F&=&\|{W_1^H\Delta B+(W_1(t)-W_1(0))^H\Delta B}\|_F\nonumber\\
&\le& \|{W_1^H\Delta B}\|_F+\|{\Delta B}\|_2\|{W_1(t)-W_1(0)}\|_F\nonumber\\
&\le& \|{W_1^H\Delta B}\|_F+\|{\Delta B}\|_2\|{W_1(t^{\star\star})-W_1(0)}\|_F,
\end{eqnarray}
where $t^\star,t^{\star\star}\in[0,\,1]$ satisfying
\begin{eqnarray*}
\|{V_1(t^\star)-V_1(0)}\|_F&=&\max_{0\le t\le 1}\|{V_1(t)-V_1(0)}\|_F=:\alpha,\\
\|{W_1(t^{\star\star})-W_1(0)}\|_F&=&\max_{0\le t\le 1}\|{W_1(t)-W_1(0)}\|_F=:\beta.
\end{eqnarray*}
Then it follows from (\ref{c-3-15}) and (\ref{c-3-16}) that
\begin{eqnarray}
\nonumber
&&\quad\|{\Delta BV_1(t)}\|_F^2+\|{W_1(t)^H\Delta B}\|_F^2\\
&&\le (\|{\Delta BV_1}\|_F+\alpha\|{\Delta B}\|_2)^2
+(\|{W_1^H\Delta B}\|_F+\beta\|{\Delta B}\|_2)^2\nonumber\\
\nonumber
&&\le \|{\Delta BV_1}\|_F^2+\|{W_1^H\Delta B}\|_F^2+2\|{\Delta B}\|_2
(\alpha\|{\Delta BV_1}\|_F+\beta \|{W_1^H\Delta B}\|_F)+\|{\Delta B}\|_2^2(\alpha^2+\beta^2)\\
\nonumber
&&\le \|{\Delta BV_1}\|_F^2+\|{W_1^H\Delta B}\|_F^2+2\|{\Delta B}\|_2
\sqrt{\|{\Delta BV_1}\|_F^2+\|{W_1^H\Delta B}\|_F^2}\sqrt{\alpha^2+\beta^2}+\|{\Delta B}\|_2^2(\alpha^2+\beta^2)\\
\label{udtbd}
&&=\left(\sqrt{\|{\Delta BV_1}\|_F^2+\|{W_1^H\Delta B}\|_F^2}
+\|{\Delta B}\|_2\sqrt{\alpha^2+\beta^2}\right)^2.
\end{eqnarray}
Using (\ref{ldtbd}), (\ref{udtbd}) and $\wh\rho_{1,\min}\le \rho_{1,\min}$, we obtain
\begin{eqnarray*}
\wh\rho_{1,\min}^2(\alpha^2+\beta^2)&\le& \rho_{1,\min}^2\alpha^2+\wh\rho^2_{1,\min}\beta^2\nonumber\\
&=&\wh\rho_{1,\min}^2\|{W_1(t^{\star\star})-W_1(0)}\|_F^2
+\rho_{1,\min}^2\|{V_1(t^\star)-V_1(0)}\|_F^2\nonumber\\
&=&\wh\rho_{1,\min}^2\left\|\int_0^{t^{\star\star}}\dot W_1(t){\rm d}t\right\|_F^2
+\rho_{1,\min}^2\left\|\int_0^{t^\star}\dot V_1(t){\rm d}t\right\|_F^2\nonumber\\
&\le&\wh\rho_{1,\min}^2\int_0^{t^{\star\star}}\|{\dot W_1(t)}\|_F^2{\rm d}t
+\rho_{1,\min}^2\int_0^{t^{\star}}\|{\dot V_1(t)}\|_F^2{\rm d}t\nonumber\\
&\le&\int_0^{1}(\wh\rho_{1,\min}^2\|{\dot W_1(t)}\|_F^2
+\rho_{1,\min}^2\|{\dot V_1(t)}\|_F^2){\rm d}t\nonumber\\
&\le&\int_0^{1}(\|{\Delta BV_1(t)}\|_F^2+\|{W_1(t)^H\Delta B}\|_F^2){\rm d}t \nonumber\\
&\le&\left(\sqrt{\|{\Delta BV_1}\|_F^2+\|{W_1^H\Delta B}\|_F^2}
+\|{\Delta B}\|_2\sqrt{\alpha^2+\beta^2}\right)^2.
\end{eqnarray*}
By taking the square root on both sides of the inequality, simple calculations yield
\begin{eqnarray}\label{c-3-18}
\sqrt{\alpha^2+\beta^2}\le \frac{\sqrt{\|{\Delta BV_1}\|_F^2+\|{W_1^H\Delta B}\|_F^2}}
{\wh\rho_{1,\min}-\|{\Delta B}\|_2}.
\end{eqnarray}
By (\ref{udtbd}) and (\ref{c-3-18}), we get
\begin{eqnarray}\label{c-3-19}
\|{\Delta BV_1(t)}\|_F^2+
\|{W_1(t)^H\Delta B}\|_F^2\le\left(\frac{\wh \rho_{1,\min}}{\wh \rho_{1,\min}-\|{\Delta B}\|_2}\right)^2(\|{\Delta BV_1}\|_F^2+\|{W_1^H\Delta B}\|_F^2).
\end{eqnarray}
Then by using (\ref{ldtbd}) and (\ref{c-3-19}), we have
\eqn
\nonumber
&&2\|{\wt\Sigma_1-\Sigma_1}\|_F^2+\wh \rho_{1,\min}^2\|{\wt W_1-\wt W_1}\|_F^2+
\rho_{1,\min}^2\|{\wt V_1-V_1}\|_F^2\\
\nonumber
&&\qquad\qquad \le \int_0^1\left(2\|{\dot\Sigma_1(t)}\|_F^2+\wh\rho_{1,\min}^2\|{\dot W_{1}(t)}\|_F^2+
\rho_{1,\min}^2\|{\dot V_{1}(t)}\|_F^2\right){\rm d}t\\
\label{lint}
&&\qquad\qquad\le\int_0^1\left(\|{\Delta BV_1(t)}\|_F^2+
\|{W_1(t)^H\Delta B}\|_F^2\right){\rm d}t \\
\nonumber
&&\qquad\qquad\le \left(\frac{\wh \rho_{1,\min}}{\wh \rho_{1,\min}-\|{\Delta B}\|_2}\right)^2
(\|{\Delta BV_1}\|_F^2+\|{W_1^H\Delta B}\|_F^2),
\enn
which leads to (\ref{c-3-14-2}) by using the bound (\ref{c-3-0}).

When $m=n$, one has $\wh \rho_{1,\min}=\rho_{1,\min}$, and the bound (\ref{c-3-14-2}) reduces to  (\ref{c-3-14-3}).
\end{proof}

\medskip

\begin{Corollary}\label{cor31}
Under the assumptions of Theorem \ref{them32}, we have
{\footnotesize\begin{eqnarray*}
&&2\|{{\rm Sing}^{\downarrow}(\wt\Sigma_1)-{\rm Sing}^{\downarrow}(\Sigma_1)}\|_F^2+\wh \rho_{1,\min}^2\|{\sin\Theta(W_1,\wt W_1)}\|_F^2+
\rho_{1,\min}^2\|{\sin\Theta(V_1, \wt V_1}\|_F^2\nonumber\\
& \le&
\left(\frac{\wh \rho_{1,\min}}{\wh \rho_{1,\min}-\|{\Delta B}\|_2}\right)^2\left(\|{\Delta BV_1}\|_F^2+\|{W_1^H\Delta B}\|_F^2\right).
\end{eqnarray*}}
When $m=n$,  we have
{\footnotesize\begin{eqnarray*}
 &&2\left(1-\frac{\|{\Delta B}\|_2}{\rho_{1,\min}}\right)^2
\|{{\rm Sing}^{\downarrow}(\wt\Sigma_1)-{\rm Sing}^{\downarrow}(\Sigma_1)}\|_F^2+(\rho_{1,\min}-\|\Delta B\|_2)^2\left(\|{\sin\Theta(W_1,\wt W_1)}\|_F^2+
\|{\sin\Theta(V_1, \wt V_1)}\|_F^2\right)\\
& \le&
\|{\Delta B V_1}\|_F^2+\|{W_1^H\Delta B}\|_F^2.
\end{eqnarray*}}
\end{Corollary}
\begin{proof}
It can be proved in the same way as that for Corollary~\ref{cor1}
\end{proof}\vskip 2mm

\begin{Remark}\label{rem351} The proof follows by similar discussions to Remark 2.2.
In Theorem \ref{them32} and Corollary \ref{cor31}, our combined bounds require $\|\Delta B\|_2<\wh\rho_{1,\min}$ or $\|\Delta B\|_2<\rho_{1,\min}$ (when $m=n$). Hence these bounds are local. A sufficient condition for $\|\Delta B\|_2<\wh\rho_{1,\min}$ is
$$\|\Delta B\|_2<\min\{\rho_1(0)/3,~\sigma_{1,\min}(0)/2\}$$
and a sufficient condition for $\|\Delta B\|_2<\rho_{1,\min}$ is
$\rho_1(0)>3\|\Delta B\|_2$.
\end{Remark}
\vskip 2mm

\begin{Remark}\label{rem361}
Applying the Mean Value Theorem to the second integral in \eqref{lint}, we have the following
simpler bounds,
\bstar
&& 2\|{\rm Sing}^{\downarrow}(\wt\Sigma_1)-{\rm Sing}^{\downarrow}(\Sigma_1)\|_F^2+\wh \rho_{1,\min}^2\|{\wt W_1-\wt W_1}\|_F^2+
\rho_{1,\min}^2\|{\wt V_1-V_1}\|_F^2\\
&\le&\|{\Delta BV_1(t_0)}\|_F^2+\|{W_1(t_0)^H\Delta B}\|_F^2,\qquad\qquad\qquad\qquad\qquad\qquad\qquad\quad\quad\mbox{when } m>n;\\
&& 2\|{\rm Sing}^{\downarrow}(\wt\Sigma_1)-{\rm Sing}^{\downarrow}(\Sigma_1)\|_F^2+ \rho_{1,\min}^2(\|{\wt W_1-\wt W_1}\|_F^2+
\|{\wt V_1-V_1}\|_F^2)\\
&\le&\|{\Delta BV_1(t_0)}\|_F^2+\|{W_1(t_0)^H\Delta B}\|_F^2,\qquad\qquad\qquad\qquad\qquad\qquad\qquad\quad\quad\mbox{when } m=n,
\estar
for some $t_0\in [0,\,1]$.
\end{Remark}

\medskip

\section{Conclusion}
By using a specific analytic decomposition, we obtain a combined bound for perturbations of the  eigenspaces of a Hermitian matrix that form a direct sum of the entire vector space
and all the eigenvalues. Combined bounds for a single eigenspace and its corresponding eigenvalues are also provided. The bounds are similar to the existing ones
in \cite{li07} but potentially sharper.   Elementary and simple calculus tools are employed for deriving the bounds. The same types of combined perturbation bounds are also derived for the left and right singular subspaces and singular values of a general matrix.


\end{document}